\documentclass[10pt,reqno]{amsart}
\usepackage{floatflt}
\usepackage{xypic}
\usepackage{rotating}
\usepackage{bbm}
\usepackage{mathrsfs}
\usepackage{diagbox}
\usepackage{dutchcal}
\usepackage{extarrows}
\usepackage{cite}
\usepackage{amsfonts} 
\usepackage[dvipsnames,usenames]{color}
\usepackage{graphicx,latexsym,bm,amsmath,amssymb,verbatim,multicol,lscape}
\makeatletter
\@namedef{subjclassname@2020}{\textup{2020} Mathematics Subject Classification}
\makeatother
\newtheorem{theorem}{Theorem} [section]

\newtheorem{conjecture}[theorem]{Conjecture}
\newtheorem{lemma}[theorem]{Lemma}

\numberwithin{equation}{section}

\def\bew{\begin{widetext}}
\def\eew{\end{widetext}}
\def\be{\begin{equation}}
\def\ee{\end{equation}}
\def\bea{\begin{eqnarray}}
\def\eea{\end{eqnarray}}

\allowdisplaybreaks
\begin{document}
\title{Proofs of Lupu's conjectures for multiple zeta values and multiple $t$-values}
\begin{abstract}
Let $r\ge 1$ be an integer. For any multiple index $\mathbf{s}=(s_1,s_2,\cdots,s_r)
\in\mathbb{Z}_{\geq 1}^r$ with $s_r\textgreater 1$, the multiple zeta value
(MZV for short) is defined by
\begin{align*}
\zeta(s_1,s_2,\cdots,s_r):=\sum_{1\leq k_1<k_2<\cdots<k_r}
\frac{1}{k_1^{s_1}k_2^{s_2}\cdots k_r^{s_r}}
\end{align*}
and the multiple $t$-value is defined by
\begin{align*}
t(s_1,s_2,...,s_r):=\sum_{1\leq k_1<k_2<...<k_r}
\frac{1}{(2k_1-1)^{s_1}(2k_2-1)^{s_2}...(2k_r-1)^{s_r}},
\end{align*}
where if the index is empty, then we define the value $t(\emptyset):=1$.
We denote by $\{a_1,\cdots,a_k\}^d$ the sequence formed by repeating the sequence
$\{a_1,\cdots,a_k\}$ exactly $d$ times. Let $H(a,b)=\zeta(\{2\}^a,3,\{2\}^b)$
and $T(a,b):=t(\{2\}^a,3,\{2\}^b)$. In this paper, by using the Lai-Lupu-Orr
integral expressions for $H(a,b)$ and $T(a,b)$ and the properties of
Beta function and Gamma function, we show that for any nonnegative integers
$a$ and $b$, we have
\begin{align*}
H(a,b):=\frac{-4\pi^{2a+2b+2}}{(2a+2)!}\sum_{n=0}^{\infty}
\frac{\zeta(2n)}{(2n+2a+2)(2n+2a+3)\cdots(2n+2a+2b+3)2^{2n}}
\end{align*}
and
\begin{align*}
T(a,b)=\frac{-2}{(2a+1)!}\left(\frac{\pi}{2}\right)^{2a+2b+2}
\sum_{n=0}^{\infty}\frac{\zeta(2n)}{(2n+2a+1)(2n+2a+2)\cdots(2n+2a+2b+2)2^{2n}}.
\end{align*}
This confirms two conjectures of Lupu proposed in [C. Lupu, Another look at Zagier's
formula for multiple zeta values involving Hoffman elements, Math. Z. 301 (2022), 3127-3140].
\end{abstract}

\author[W.Z. Lei]{Wenzhong Lei}
\address{Mathematical College, Sichuan University,
Chengdu 610064, P.R. China}
\email{lwzh1729@163.com}
\author[J.M. Yu]{Jinmin Yu}
\address{Mathematical College, Sichuan University,
Chengdu 610064, P.R. China}
\email{jmyumath@163.com}
\author[S.F. Hong]{Shaofang Hong\textsuperscript{*}}
\address{Mathematical College, Sichuan University,
Chengdu 610064, P.R. China}
\email{sfhong@scu.edu.cn}
\thanks{$^*$S.F. Hong is the corresponding author and
was supported partially by National Science Foundation
of China \#12571007.}
\maketitle

\section{Introduction}
The systematic study of multiple zeta values (MZVs) started
in the early 1990s with the work of Hoffman \cite{[H-PJM]}
and Zagier \cite{[Z-FECM]}. Let $r\ge 1$ be an integer. For any multiple index
$\mathbf{s}=(s_1,s_2,\cdots,s_r)\in \mathbb{Z}_{\geq 1}^r$ with
$s_r\textgreater 1$, the {\it multiple zeta value (MZV for short)}
is defined by
\begin{align*}
\zeta(s_1,s_2,\cdots,s_r):=\sum_{1\leq k_1<k_2<\cdots<k_r}
\frac{1}{k_1^{s_1}k_2^{s_2}\cdots k_r^{s_r}}.
\end{align*}
The {\it weight} of MZV is defined to be the number
$|\mathbf{s}|:=s_1+\cdots+s_r$, and its {\it depth} is defined
by $l(\mathbf{s}):=r$. For brevity, we let $\{a_1,\cdots,a_k\}^d$
denote the sequence formed by repeating the sequence
$\{a_1,\cdots,a_k\}$ exactly $d$ times. The simplest precise
evaluations of MZVs and MZSVs are given by
\begin{eqnarray*}
\zeta(\{2\}^d)=\frac{\pi^{2d}}{(2d+1)!}.
\end{eqnarray*}

A central problem in MZV is to investigate the algebraic structure
of the $\mathbb{Q}$-vector space $\mathfrak{Z}_k$ spanned by MZVs
of weight $k$. In this direction, Hoffman \cite{[H-JA]} proposed a
conjecture that stats that every MZV of weight $k$ can be expressed as a
$\mathbb{Q}$-linear combination of those MZVs whose entries are only
$2$'s and $3$'s. This conjecture was proved by Brown \cite{[B-ANN]}
using motivic method. In particular, Brown showed that
$H(a,b):=\zeta(\{2\}^a,3,\{2\}^b)$ can be expressed as a $\mathbb{Q}$-linear
combination of the products $\pi^{2m}\zeta(2n+1)$ with $m+n=a+b+1$.
In 2012, Zagier \cite{[Z-AM]} found surprisingly an elegant and explicit
formula as follows.

\begin{theorem}\label{theorem1.01}
(Zagier \cite{[Z-AM]}) For non-negative integers $a$ and $b$,
\begin{align}
H(a,b)=2\sum_{k=1}^{a+b+1}(-1)^kc_{a,b}\zeta(2k+1)
\zeta(\{2\}^{a+b+1-k}),
\end{align}
where
$$c_{a,b}:=\binom{2k}{2a+2}-\Big(1-\frac{1}{2^{2k}}\Big)\binom{2k}{2b+1}.$$
\end{theorem}
\noindent The proof of Theorem \ref{theorem1.01} is given by computing the associated
generating functions of both sides in a closed form, and then showing
they are entire functions of exponential growth that agree at
sufficiently many points to force their equality. Later on, his proof
was simplified by Li \cite{[L-MRL]}. In 2017, Hessami Pilehrood
and Hessami Pilehrood \cite{[HHT-JMAA]} gave an alternative proof of Theorem
\ref{theorem1.01}.

In a similar way, Hoffman \cite{[H-CNTP]} defined the multiple
$t$-values  (``odd variant'' of MZVs). For any multiple index
$\mathbf{s}=(s_1,s_2,...,s_r)\in \mathbb{Z}_{\geq 1}^r$, the
{\it multiple $t$-value } is defined by
\begin{align*}
t(s_1,s_2,...,s_r):=\sum_{1\leq k_1<k_2<...<k_r}
\frac{1}{(2k_1-1)^{s_1}(2k_2-1)^{s_2}...(2k_r-1)^{s_r}},
\end{align*}
where if the index is empty, then we define the value $t(\emptyset):=1$.
As we stated above, first let us recall the the analogue formula
for the simplest multiple zeta value (\cite{[MK-PSPM]},\cite{[Z-FM]})
\begin{align*}
t(\mathop{\underbrace{2,2,...,2}}_{n})=\frac{\pi^{2n}}{2^{2n}(2n)!}.
\end{align*}

By the method of establishing the generating functions, Murakami \cite{[M-MA]}
obtained an equivalent version of Theorem \ref{theorem1.01}. Define
$$\widetilde{t}(k_1,k_2,...,k_r):=2^{|k|}t(k_1,k_2,...,k_r)$$
and set
$$K(a,b):=\widetilde{t}(\{2\}^a,3,\{2\}^b),K(n)=\widetilde{t}(\{2\}^n)$$
for any $a,b,n\in \mathbb{N}_{\geq 0}$. Then for all integers $a,b\geq 0$,
Murakami \cite{[M-MA]} provided the following formula
\begin{align}\label{eq1.1}
K(a,b)=\sum_{k=1}^{a+b+1}(-1)^{k-1}d_{a,b}K(a+b-k+1)\zeta(2k+1),
\end{align}
where
$$d_{a,b}:=\binom{2k}{2a+1}+\binom{2k}{2b+1}\Big(1-\frac{1}{2^{2k}}\Big).$$
Moreover, let $T(a,b):=t(\{2\}^a,3,\{2\}^b)$. Then \eqref{eq1.1} is equivalent to
\begin{align}
T(a,b)=2\sum_{k=1}^{r+s+1}(-1)^{k-1}d_{a,b}\frac{1}{2^{2k}}\zeta(2k+1)
t(\{2\}^{a+b-k+1}),
\end{align}

In a different approach, Lupu \cite{[L-MZ]} provided elementary proofs of
Zagier's formula for the special case $b=0$, and also established a Zagier-type
formula for multiple $t$-values when $b=0$. For the general case $b>0$, Lupu
\cite{[L-MZ]} proposed the following two conjectures.

\begin{conjecture}\label{con1}
For nonnegative integers $a$ and $b$, we have
\begin{align}
H(a,b)=\frac{-4\pi^{2a+2b+2}}{(2a+2)!}\sum_{n=0}^{\infty}
\frac{\zeta(2n)}{(2n+2a+2)(2n+2a+3)\cdots(2n+2a+2b+3)2^{2n}}.
\end{align}
\end{conjecture}

\begin{conjecture}\label{con2}
For nonnegative integers $a$ and $b$, we have
\begin{align}
T(a,b)=\frac{-2}{(2a+1)!}\left(\frac{\pi}{2}\right)^{2a+2b+2}
\sum_{n=0}^{\infty}\frac{\zeta(2n)}{(2n+2a+1)(2n+2a+2)\cdots(2n+2a+2b+2)2^{2n}}.
\end{align}
\end{conjecture}
\noindent Notice that Conjectures 1.2 and 1.3 are shown in \cite{[L-MZ]}
to be true when $b=0$. But for the general case $b>0$, these two conjectures
are still kept open so far.

In 2026, Lai, Lupu, and Orr \cite{[LLO-PAMS]} gave elementary and direct proofs
for both $H(a, b)$ and $T(a, b)$. In fact, they \cite{[LLO-PAMS]} first provided
an integral expression for $H(a, b)$, and then showed that this integral expression
is equal to Zagier's result. The same proof method was applied to $T(a, b)$.

In this paper, our main goal is to investigate Conjectures \ref{con1} and \ref{con2}.
We will prove that both conjectures of Lupu are true. Actually, by employing the integral
expressions presented in \cite{[LLO-PAMS]} for $H(a,b)$ and $T(a,b)$ and using
the properties of Beta and Gamma functions, we show the following main results
of this paper.

\begin{theorem}\label{theorem1.1}
For nonnegative integers $a$ and $b$, we have $H(a,b)=L(a,b)$, where
\begin{align*}
L(a,b):=\frac{-4\pi^{2a+2b+2}}{(2a+2)!}\sum_{n=0}^{\infty}
\frac{\zeta(2n)}{(2n+2a+2)(2n+2a+3)\cdots(2n+2a+2b+3)2^{2n}}.
\end{align*}
\end{theorem}

\begin{theorem}\label{theorem1.2}
For nonnegative integers $a$ and $b$, we have $T(a,b)=\hat L(a,b)$, where
\begin{align*}
\hat L(a,b):=\frac{-2}{(2a+1)!}\left(\frac{\pi}{2}\right)^{2a+2b+2}
\sum_{n=0}^{\infty}\frac{\zeta(2n)}{(2n+2a+1)(2n+2a+2)\cdots(2n+2a+2b+2)2^{2n}}.
\end{align*}
\end{theorem}

This paper is organized as follows. In Section 2, we supply several preliminary lemmas
that are needed in the proofs of our main results. Sections 3 and 4 are devoted to
the proofs of Theorems \ref{theorem1.1} and \ref{theorem1.2}, respectively.

\section{Preliminaries}

In this section, we give some preliminary lemmas that are needed in the proofs
of our main results.

As usual, for any real numbers $r$ and $s$, we define the {\it Beta function},
denoted by $B(r,s)$, and the {\it Gamma function}, denoted by $\Gamma(r)$, as follows:
$$
B(r,s)=\int_{0}^{1} x^{r-1}(1-x)^{s-1}\mathrm{d}x
$$
and
$$
\Gamma(r)=\int_{0}^{\infty}e^{-x} x^{r-1}\mathrm{d}x.
$$
See, for instance, Definition 1.1.1 of \cite{[AAR]}. It is well known that
the Beta function and the Gamma function hold the following properties.

\begin{lemma}\label{lemma2.0} (Theorem 1.1.4 of \cite{[AAR]})
We have
$$
B(r,s)=\frac{\Gamma(r)\Gamma(s)}{\Gamma(r+s)}
$$
and
$$
\Gamma(r)=(r-1)!,
$$
where $r$ an $s$ are positive integers.
\end{lemma}

We also need the following integral expressions for $H(a,b)$ and $T(a,b)$ which
are due to Lai, Lupu and Orr \cite{[LLO-PAMS]}.

\begin{lemma}\label{lemma2.1} (Theorem 3.2 of \cite{[LLO-PAMS]})
We have
\begin{align*}
H(a,b) = \frac{\pi^{2b}2^{2a + 3}}{(2a + 2)!(2b + 1)!}\int_{0}^{\frac{\pi}{2}} x^{2a + 2}
\left(1 - \frac{2}{\pi} x\right)^{2b+1}\cot(x)\mathrm{d}x.
\end{align*}
\end{lemma}

\begin{lemma}\label{lemma2.2} (Theorem 4.2 of \cite{[LLO-PAMS]})
We have
\begin{align*}
T(a,b) = \frac{\pi^{2b+1}}{2^{2b+1}(2a+1)!(2b+1)!} \int_{0}^{\pi/2} x^{2a+1}
\left(1 - \frac{2x}{\pi}\right)^{2b+1} \cot(x) \mathrm{d}x.
\end{align*}
\end{lemma}

\section{Proof of Theorem \ref{theorem1.1}}

In this section, we use the lemmas presented in the  previous
section to show Theorem \ref{theorem1.1}.\\

\noindent {\it Proof of Theorem \ref{theorem1.1}.}
From Lemma \ref{lemma2.1} we know that
\begin{align}\label{eq3.1}
H(a,b)=&\frac{\pi^{2b}2^{2a + 3}}{(2a + 2)!(2b + 1)!}\int_{0}^{\frac{\pi}{2}} x^{2a + 2}
\left(1 - \frac{2}{\pi} x\right)^{2b + 1}\cot(x)\mathrm{d}x.
\end{align}

First, we consider the integral
$$\int_{0}^{\frac{\pi}{2}} x^{2a + 2}\left(1 - \frac{2}{\pi} x\right)^{2b + 1} \cot(x)\mathrm{d}x.$$
By using the well-known identity (see, for example, Theorem 1.2.4 of \cite{[AAR]})
\begin{align}\label{eq3.1'}
x\cot(x)=-2\sum_{n=0}^{\infty}\frac{\zeta(2n)}{\pi^{2n}}x^{2n},
\end{align}
we derive that
\begin{align}\label{eq3.2}
&\int_{0}^{\frac{\pi}{2}}x^{2a+2}\Big(1-\frac{2}{\pi} x\Big)^{2b+1}
\cot(x)\mathrm{d}x\nonumber\\
=&-2\int_{0}^{\frac{\pi}{2}}x^{2a+1}\Big(1-\frac{2}{\pi} x\Big)^{2b+1}
\sum_{n=0}^{\infty}\frac{\zeta(2n)}{\pi^{2n}}x^{2n} \mathrm{d}x \nonumber\\
=&-2\sum_{n=0}^{\infty}\frac{\zeta(2n)}{\pi^{2n}}
\int_{0}^{\frac{\pi}{2}}x^{2n+2a+1}\Big(1-\frac{2}{\pi} x\Big)^{2b+1}\mathrm{d}x.
\end{align}

Second, we consider the integral
$$\int_{0}^{\frac{\pi}{2}}x^{2n+2a+1}\Big(1-\frac{2}{\pi} x\Big)^{2b+1} \, \mathrm{d}x$$
appeared in \eqref{eq3.2}. Changing the variable $x\mapsto \frac{\pi}{2}t$ yields that
\begin{align*}
& \int_{0}^{\frac{\pi}{2}}x^{2n+2a+1}\Big(1-\frac{2}{\pi} x\Big)^{2b+1}\mathrm{d}x\\
=&\int_{0}^{1} \Big(\frac{\pi}{2}\Big)^{2n+2a+2}t^{2n+2a+1}(1-t)^{2b+1}
\, \mathrm{d}t \nonumber\\
=&\Big(\frac{\pi}{2}\Big)^{2n+2a+2}\int_{0}^{1} t^{2n+2a+1}(1-t)^{2b+1}\mathrm{d}t.
\end{align*}
Now by Lemma 2.1, one has
\begin{align*}
&\int_{0}^{1} t^{2n+2a+1}(1-t)^{2b+1}\mathrm{d}t\\
=&B(2n+2a+2,2b+2)\\
=& \frac{\Gamma(2n+2a+2)
\Gamma(2b+2)}{\Gamma(2n+2a+2b+4)}\\
=& \frac{(2n+2a+1)!(2b+1)!}{(2n+2a+2b+3)!}.
\end{align*}
Then we can conclude that
\begin{align}\label{eq3.3}
\int_{0}^{\frac{\pi}{2}}x^{2n+2a+1}\Big(1-\frac{2}{\pi} x\Big)^{2b+1}\mathrm{d}x
=&\Big(\frac{\pi}{2}\Big)^{2n+2a+2}\frac{(2n+2a+1)!(2b+1)!}{(2n+2a+2b+3)!}.
\end{align}
Substituting \eqref{eq3.3} into \eqref{eq3.2}, we have
\begin{align}\label{eq3.4}
&\int_{0}^{\frac{\pi}{2}}x^{2a+2}\Big(1-\frac{2}{\pi} x\Big)^{2b+1}
\cot(x)\mathrm{d}x\nonumber\\
=&-2(2b+1)!\sum_{n=0}^{\infty}\frac{\zeta(2n)}{\pi^{2n}}
\Big(\frac{\pi}{2}\Big)^{2n+2a+2}
\frac{(2n+2a+1)!}{(2n+2a+2b+3)!}\nonumber\\
=&-\frac{2\pi^{2a+2}(2b+1)!}{2^{2a+2}}\sum_{n=0}^{\infty}
\frac{\zeta(2n)(2n+2a+1)!}{2^{2n}(2n+2a+2b+3)!}\nonumber\\
=&-\frac{\pi^{2a+2}(2b+1)!}{2^{2a+1}}\sum_{n=0}^{\infty}
\frac{\zeta(2n)(2n+2a+1)!}{2^{2n}(2n+2a+2b+3)!}.
\end{align}
Finally, putting \eqref{eq3.4} into \eqref{eq3.1}, one gets that
\begin{align*}
& H(a,b)\\
=&-\frac{\pi^{2b}2^{2a+3}}{(2a+2)!(2b+1)!}
\frac{\pi^{2a+2}(2b+1)!}{2^{2a+1}}\sum_{n=0}^{\infty}
\frac{\zeta(2n)(2n+2a+1)!}{2^{2n}(2n+2a+2b+3)!}\nonumber\\
=&-\frac{4\pi^{2a+2b+2}}{(2a+2)!}\sum_{n=0}^{\infty}
\frac{\zeta(2n)}{2^{2n}(2n+2a+2)(2n+2a+3)\ldots(2n+2a+2b+3)}
\end{align*}
as required.

This completes the proof of Theorem \ref{theorem1.1}.
\hfill$\Box$\\

\section{Proof of Theorem \ref{theorem1.2}}
In this section, we use the lemmas presented in Section 2 to give the
proof of Theorem \ref{theorem1.2}.\\

\noindent {\it Proof of Theorem \ref{theorem1.2}.}
From Lemma \ref{lemma2.2}, we know that
\begin{align}\label{eq3.5}
T(a,b)=\frac{\pi^{2b+1}}{2^{2b+1}(2a+1)!(2b+1)!}
\int_{0}^{\pi/2} x^{2a+1} \Big(1-\frac{2x}{\pi}\Big)^{2b+1}\cot(x) \, \mathrm{d}x.
\end{align}

First of all, we compute the integral
$$\int_{0}^{\frac{\pi}{2}} x^{2a+1}\Big(1 - \frac{2}{\pi} x\Big)^{2b+1}\cot(x)\mathrm{d}x.$$
By making use of the identity (3.2), we derive that
\begin{align}\label{eq3.6}
&\int_{0}^{\frac{\pi}{2}}x^{2a+1}\Big(1-\frac{2}{\pi} x\Big)^{2b+1}\cot(x)\mathrm{d}x\nonumber\\
=&-2\int_{0}^{\frac{\pi}{2}}x^{2a}\Big(1-\frac{2}{\pi} x\Big)^{2b+1}
\sum_{n=0}^{\infty}\frac{\zeta(2n)}{\pi^{2n}}x^{2n}\mathrm{d}x \nonumber\\
=&-2\sum_{n=0}^{\infty}\frac{\zeta(2n)}{\pi^{2n}}
\int_{0}^{\frac{\pi}{2}}x^{2n+2a}\Big(1-\frac{2}{\pi} x\Big)^{2b+1}\mathrm{d}x.
\end{align}

In what follows, we calculate the integral
$$\int_{0}^{\frac{\pi}{2}}x^{2n+2a}\Big(1-\frac{2}{\pi} x\Big)^{2b+1}\mathrm{d}x.$$
Letting $t=\frac{2}{\pi}x$ gives that $x=\frac{\pi}{2}$ and $\mathrm{d} x=\frac{\pi}{2}\mathrm{d} t$. Thus
\begin{align}\label{eq3.7}
&\int_{0}^{\frac{\pi}{2}}x^{2n+2a}\Big(1-\frac{2}{\pi} x\Big)^{2b+1}\mathrm{d}x \nonumber\\
=&\int_{0}^{1} \Big(\frac{\pi}{2}\Big)^{2n+2a+1}t^{2n+2a}(1-t)^{2b+1} \mathrm{d}t \nonumber\\
=&\Big(\frac{\pi}{2}\Big)^{2n+2a+1}\int_{0}^{1} t^{2n+2a}(1-t)^{2b+1}\mathrm{d}t. \nonumber\\
\end{align}
Now applying Lemma 2.1 tells us that
\begin{align}\label{eq3.7'}
&\int_{0}^{1} t^{2n+2a}(1-t)^{2b+1}\mathrm{d}t\\
=& B(2n+2a+1,2b+2)\nonumber\\
=& \frac{\Gamma(2n+2a+1)\Gamma(2b+2)}{\Gamma(2n+2a+2b+3)}\nonumber\\
=&\frac{(2n+2a)!(2b+1)!}{(2n+2a+2b+2)!}.
\end{align}
Hence
\begin{align}\label{eq3.7}
&\int_{0}^{\frac{\pi}{2}}x^{2n+2a}\Big(1-\frac{2}{\pi} x\Big)^{2b+1}\mathrm{d}x
=\Big(\frac{\pi}{2}\Big)^{2n+2a+1}\frac{(2n+2a)!(2b+1)!}{(2n+2a+2b+2)!}.
\end{align}
Substituting \eqref{eq3.7} into \eqref{eq3.6}, we have
\begin{align}\label{eq4.6}
&\int_{0}^{\frac{\pi}{2}}x^{2a+1}\Big(1-\frac{2}{\pi} x\Big)^{2b+1}\cot(x) \mathrm{d}x\nonumber\\
=&-2\sum_{n=0}^{\infty}\frac{\zeta(2n)}{\pi^{2n}}\Big(\frac{\pi}{2}\Big)^{2n+2a+1}
\frac{(2n+2a)!(2b+1)!}{(2n+2a+2b+2)!}\nonumber\\
=&-\frac{2\pi^{2a+1}(2b+1)!}{2^{2a+1}}
\sum_{n=0}^{\infty}\frac{\zeta(2n)(2n+2a)!}{2^{2n}(2n+2a+2b+2)!}\nonumber\\
=&-\frac{\pi^{2a+1}(2b+1)!}{2^{2a}}
\sum_{n=0}^{\infty}\frac{\zeta(2n)(2n+2a)!}{2^{2n}(2n+2a+2b+2)!}.
\end{align}
Therefore by (4.1) and (4.6), one deduces that
\begin{align*}
& T(a,b)\nonumber\\
=&\frac{\pi^{2b+1}}{2^{2b+1}(2a+1)!(2b+1)!}\int_{0}^{\frac{\pi}{2}} x^{2a + 11}
\left(1 - \frac{2}{\pi} x\right)^{2b+1}\cot(x)\mathrm{d}x\nonumber\\
=&-\frac{\pi^{2b+1}}{2^{2b+1}(2a+1)!(2b+1)!}
\frac{\pi^{2a+1}(2b+1)!}{2^{2a}}
\sum_{n=0}^{\infty}\frac{\zeta(2n)(2n+2a)!}{2^{2n}(2n+2a+2b+2)!}\nonumber\\
=&-\frac{2}{(2a+1)!}\Big(\frac{\pi}{2}\Big)^{2a+2b+2}
\sum_{n=0}^{\infty}\frac{\zeta(2n)}{(2n+2a+1)(2n+2a+2)\ldots(2n+2a+2b+3)2^{2n}}
\end{align*}
as desired.

This concludes the proof of Theorem \ref{theorem1.2}. \hfill$\Box$\\

 \end{document}